\documentclass[journal]{IEEEtran}

\ifCLASSINFOpdf
\else
\fi

\usepackage[caption=false,font=normalsize,labelfont=sf,textfont=sf]{subfig}
\usepackage{stfloats}
\usepackage{url}
\usepackage{verbatim}
\usepackage{cite} 
\usepackage{amsmath,amssymb,amsthm}
\usepackage{amsfonts,amsbsy,bm}
\usepackage{latexsym,mathrsfs}

\usepackage{array}
\usepackage{booktabs}
\usepackage{tabularx}
\usepackage{makecell}
\usepackage[para]{threeparttable}
\usepackage{float}

\usepackage{graphicx}
\usepackage{xcolor}
\usepackage{tikz}

\usepackage{algorithm}
\usepackage{algorithmicx}
\usepackage{algpseudocode}
\usepackage{listings}

\usepackage{enumitem}
\usepackage{scalefnt}
\usepackage{manyfoot}
\usepackage{textcomp} 
\usepackage[colorlinks=true, citecolor=blue]{hyperref} 
\makeatletter

\newtheorem{case}{Case}
\newtheorem{theorem}{Theorem}
\newtheorem{example}{Example}
\newtheorem{definition}{Definition}

\newtheorem{proposition}{Proposition}
\newtheorem{corollary}{Corollary}
\newtheorem{lemma}{Lemma}

\renewcommand\thetable{\Roman{table}}

\renewcommand{\le}{\leqslant}
\renewcommand{\leq}{\leqslant}
\renewcommand{\ge}{\geqslant}
\renewcommand{\geq}{\geqslant}

\begin{document}
%
\title{A Generic Construction of $q$-ary Near-MDS Codes Supporting 2-Designs with Lengths Beyond $q + 1$}
%
%
%

	\author{Hengfeng Liu, Chunming Tang,~\IEEEmembership{Member,~IEEE,} Zhengchun Zhou,~\IEEEmembership{Senior Member,~IEEE,} Dongchun Han, Hao Chen
	\thanks{Manuscript received November 26, 2025; revised March 2, 2026; accepted April 2, 2026. This work was supported in part by the National Natural Science Foundation of China under Projects 12231015,  12571578, 62032009, and W2531009; and in part by Science and Technology Projects of Xizang Autonomous Region, China  (No. XZ202502JD0036). (\textit{Corresponding Author: Chunming Tang})}
	\thanks{Hengfeng Liu, Dongchun Han are with the School of Mathematics, Southwest Jiaotong University, Chengdu, 610031, China, and also with the
		State Key Laboratory of Integrated Service Networks, School of Cyber Engineering, Xidian University, Xi’an 710071, China  (e-mail: hengfengliu@163.com; han-qingfeng@163.com).}
	\thanks{Chunming Tang, Zhengchun Zhou are with the School of Information Science and Technology, Southwest Jiaotong University, Chengdu, 610031, China (e-mail: tangchunmingmath@163.com; zzc@swjtu.edu.cn).}
	\thanks{Hao Chen is with the College of Information Science and Technology, Jinan University, Guangzhou, Guangdong, 510632, China (e-mail: haochen@jnu.edu.cn).}	
}

%
%

\markboth{Journal of \LaTeX\ Class Files,~Vol.~14, No.~8, August~2015}%
{Shell \MakeLowercase{\textit{et al.}}: Bare Demo of IEEEtran.cls for IEEE Journals}
%



\maketitle

	\begin{abstract}
	A linear code with parameters $[n, k, n - k + 1]$ is called maximum distance separable (MDS), and one with parameters $[n, k, n - k]$ is called almost MDS (AMDS). A code is near-MDS (NMDS) if both it and its dual are AMDS. NMDS codes supporting combinatorial $t$-designs have attracted growing interest, yet constructing such codes remains highly challenging. In 2020, Ding and Tang initiated the study of NMDS codes supporting 2-designs by constructing the first infinite family, followed by several other constructions for $t > 2$, all with length at most $q + 1$. Although NMDS codes can, in principle, exceed this length, known examples supporting 2-designs and having length greater than $q + 1$ are extremely rare and limited to a few sporadic binary and ternary cases. In this paper, we present the first \emph{generic construction} of $q$-ary NMDS codes supporting 2-designs with lengths \emph{exceeding $q + 1$}. Our method leverages new connections between elliptic curve codes, finite abelian groups, subset sums, and combinatorial designs, resulting in an infinite family of such codes along with their weight distributions. 
\end{abstract}


%
\IEEEpeerreviewmaketitle

	\section{Introduction}\label{introduction}

\IEEEPARstart{W}{e} start by a brief recall of linear codes and $t$-designs. Let $\mathbb{F}_q$ denote the finite field of $q$ elements, where $q$ is a prime power. An $[n,k]$ linear code $\mathcal{C}$ over the finite field $\mathbb{F}_q$ is a $k$-dimensional subspace of $\mathbb{F}_{q}^{n}$. For a codeword  $\mathbf{c} = (c_1, c_2, \ldots, c_n) \in \mathcal{C}$, its \emph{support} is defined as $\operatorname{supp}(\mathbf{c}) = \{ 1 \leq i \leq n : c_i \neq 0 \}$. The Hamming weight of $\mathbf{c}$ is $\operatorname{wt}(\mathbf{c}) = \#\operatorname{supp}(\mathbf{c})$, and the minimum distance $d$ of $\mathcal{C}$ is the minimum Hamming weight in all non-zero codewords. The \emph{dual code} of $\mathcal{C}$ is defined by $
\mathcal{C}^\perp = \left\{ \mathbf{u} \in \mathbb{F}_q^n : \langle \mathbf{u}, \mathbf{c} \rangle = 0 \text{ for all } \mathbf{c} \in \mathcal{C} \right\},$ where $\langle \cdot, \cdot \rangle$  denote the Euclidean inner product. By definition, the dimension of $\mathcal{C}^\perp$ is $n-k$, and denote by $d^\perp$ the minimum Hamming weight of $\mathcal{C}^\perp$. Let $A_{i}$ denote the number of codewords of Hamming weight $i$ in $\mathcal{C}$. The \emph{weight distribution} of $\mathcal{C}$ is the sequence $(A_0, A_1, \cdots, A_n)$, and the weight enumerator of $\mathcal{C}$ is defined by the polynomial $A(z) = 1 + A_1 z + A_2 z^2 + \cdots + A_n z^n$. Analogously, let $(A_0^\perp, A_1^\perp, \cdots, A_n^\perp)$ denotes the weight distribution of $\mathcal{C}^\perp$. Weight distributions of linear codes have been extensively studied in \cite {DT2020, D12020, DBCH2024, Fan2024, Heng Ding Zhou2018,Heng Li Du Chen2021, Heng2023-NMDS, T20204-design, Tang xiang feng 2017,Xiang2022, Xiang2023}, they are of interest for they provide crucial information for computing the error probability of its error detection and correction \cite{T2006 }. \par

Let $n, k, t$ be positive integers with $1\le t\le k\le n$. A $t$-$\left ( n, k, \lambda   \right ) $ \emph{design} is an incidence structure $\mathcal{D}= \left (\mathcal{P} ,\mathcal{B}   \right ) $, where $\mathcal{P}$ is a set with $n$ elements and $\mathcal{B}$ is a family of $k$-subsets of $\mathcal{P}$, such that any $t$-subset of  $\mathcal{P}$ is contained in exactly $\lambda $ elements in $\mathcal{B}$. The elements in $\mathcal{P}$ are referred to as \emph{points} and those in $\mathcal{B}$ are called \emph{blocks}. Let $b$ denote the number of blocks in $\mathcal{B}$. In a $t$-$\left ( n, k, \lambda   \right ) $ design, the parameters are interrelated and satisfy $$\binom{n}{t}\lambda _{t}= \binom{k}{t} b.$$ In addition, a $t$-$\left (n, k, \lambda_{t}   \right ) $ design is also an $i$-$\left ( v, k, \lambda_{i}  \right ) $ design for $1\le i\le t-1$, where $$\lambda_{i}=\lambda_{t}\binom{v-i}{t-i} \Big/ \binom{k-i}{t-i}.$$ 
For a $t$-$(n,k,\lambda)$ design $(\mathcal{P},\mathcal{B})$, let $\mathcal{B}^{c}$ be the family of the complement sets of the blocks in $\mathcal{B}$. Then $\left(\mathcal{P},\mathcal{B}^{c}\right)$ is a $t$-$\left(n, n-k,\lambda_{c}\right)$ design, where
$$\lambda_{c} = \lambda \binom{n-t}{k}/\binom{n-t}{k-t}.$$
The pair $\left(\mathcal{P},\mathcal{B}^{c}\right)$ is defined to be the \emph{complementary design} of $(\mathcal{P},\mathcal{B})$. Denote by $\binom{\mathcal{P}}{k}$ the set of all $k$-subsets of $\mathcal{P}$, then $(\mathcal{P}, \binom{\mathcal{P}}{k})$ is a natural $k$-$(v, k, 1)$ design, called a \emph{complete design}. A $t$-design is called \emph{simple} if it does not have repeated blocks in $\mathcal{B}$. A $t$-$(v, k, \lambda)$ design is called a \emph{Steiner system} and denoted by $S(t, k, v)$ if $t \geq 2$ and $\lambda = 1$.
For more information about $t$-designs, the reader is referred to \cite{BJL,C2010}.\par
The close relation between linear codes and $t$-designs is a central topic in both coding theory and combinatorial design. It is well known that many $t$-designs can be derived from linear codes, and also many $t$-designs may yields linear codes. A standard coding approach of constructing $t$-designs goes as the following. Denote by 
$$\mathcal{H}_{w}(\mathcal{C})=\frac{1}{q-1}\{\{\operatorname{supp}(\mathbf{c}):\operatorname{wt}(\mathbf{c})=w \text{ and } \mathbf{c} \in\mathcal{C}\}\},$$ where $\{\{\}\}$ denotes the multiset notation, and $\frac{1}{q-1} S$ denotes the multiset obtained by dividing the multiplicity of each element in the multiset $S$ by $q-1$. We say that the linear code $\mathcal{C}$ supports or holds a $t$-design if the incidence structure $\left ( \mathcal{P} (\mathcal{C} ), \mathcal{H}_{w}(\mathcal{C} )\right ) $ is a $t$-$(n,w,\lambda)$ design for some $d\le w\le n $, then the number of blocks $b$ and the parameter $\lambda$ satisfy 	$$ b=\frac{1}{q-1} A_w,\quad \lambda=\frac{\binom{w}{t}}{(q-1)\binom{n}{t}} A_w.$$  Such $t$-designs are called \emph{support designs} from linear codes. Note that these $t$-designs may be simple or may have repeated blocks. Following this approach, a lot of $t$-designs have been constructed from linear codes \cite{D12020,D32020,D2024,DBCH2024,T20204-design,W2023,XuG2022,YanQ2022}.  For further information about linear codes and $t$-designs, the reader is referred to \cite{D2022}. \par
An $[n,k,n-k+1]$ linear code is called an MDS code, MDS codes have been extensively studied for decades as their parameters attain the Singleton bound. An $[n,k,n-k]$ code is said to be almost MDS (AMDS for short). Near MDS codes were introduced by Dodunekov and Landjev \cite{Dodunekov1995} aiming at constructing good linear codes by slightly weakening the restrictions of MDS codes. If the code $\mathcal{C}$ and its dual code $\mathcal{C}^\perp$ are both AMDS, then $\mathcal{C}$ is said to be near MDS (NMDS for short).  NMDS codes are attractive as they have many nice applications in combinatorial designs \cite{D12020, DT2020, Heng2023-NMDS,XuG2022}, finite geometry \cite{Dodunekov1995}, cryptography \cite{Simos2012,ZhiY2025} and
many other fields. The first NMDS code was the $[11,6,5]$ ternary Golay code \cite{G1949} , discovered by Golay in 1949. Over the years, many NMDS codes have been constructed in \cite{DY2024,Fan2024,Heng2022,Heng2023-NMDS,LiH 2023,XuL2023,XuL2024,Yin2024,ZhiY2025}, and their weight distributions were also determined.\par
Constructing $t$-designs from special NMDS codes has attracted significant attention in recent years. MDS codes do hold $t$-designs, but all these $t$-designs are complete and thus trivial \cite{D12020}. By weakening the condition, it is natural to ask if an NMDS code holds $t$-designs. Although a lot of $t$-designs from general linear codes were constructed, it is challenging to construct $t$-designs from NMDS codes. The first NMDS code holding a $t$-design was the  $[11,6,5]$ ternary Golay code, it holds $4$-designs. Moreover, its extended code holds a Steiner system $S(5,6,12)$.  In fact, it remained an open problem whether there exist infinite families of NMDS codes supporting $t$-design ($t\ge 2 $) for 70 years until Ding and Tang \cite{D12020} first presented two infinite families of NMDS codes with length $q+1$ and dimension $4$ over $\mathbb{F}_{q}$ supporting $3$-designs or $2$-designs in 2020. Subsequently, a few $q$-ary NMDS codes with length at most $q+1$ that support $t$-designs ($t\ge 2$) have been proposed in recent years. Heng \textit{et al.} \cite{Heng2023,Heng2023-NMDS} proposed NMDS codes over $\mathbb{F}_{2^{m}}$ with lengths $q-1$ or $q$ and fixed dimensions $3 \le k \le 6$ supporting $2$-designs or $3$-designs. Xu \textit{et al.} \cite{XuG2022} developed an infinite family of NMDS codes supporting $2$-designs, with length $q$ and fixed dimension $3$ over $\mathbb{F}_{3^{m}}$ and NMDS codes supporting $3$-designs over $\mathbb{F}_{3^{2m}}$. Further, Tang and Ding \cite{T20204-design} proposed the first infinite family of NMDS codes supporting $4$-designs, with length $q+1$ and fixed dimension $6$ over $\mathbb{F}_{2^{m}}$.  \par
Although many $q$-ary NMDS codes with lengths exceeding $q+1$ were constructed and a few NMDS codes supporting $2$-designs have been developed in recent years, $q$-ary NMDS codes whose lengths exceed $q+1$ and simultaneously support $2$-designs are really rare. The well-known MDS conjecture states that the maximum possible length of a $q$-ary MDS code is $q + 1$, except for trivial parameters in even characteristic fields \cite{Tsfasman1991}. Unlike MDS codes, the length of $q$-ary NMDS codes may exceed $q+1$. For example, $q$-ary NMDS codes with any length $n\le q+2\sqrt{q}+1$ can be constructed from algebraic curves \cite {Tsfasman1991}. Denote by $m(q,k)$ the maximum possible length of an NMDS code of dimension $k$ over $\mathbb{F}_{q}$. Then we have the bound $m(q,k)\le 2q+k$ \cite{Dodunekov1995}, and NMDS codes attain this bound are said to be \emph{extremal}. A $q$-ary extremal NMDS code has parameters $[2q+k,k,2q]$, and NMDS codes with parameters $[2q+k,k+1,2q-1]$ are said to be \emph{almost extremal}. Surprisingly, the codewords of minimal weight in any $[2q+k,k,2q]$ extremal NMDS code supports a $(k-1)$-design \cite{Dodunekov1995}, and so does its dual code (see Section \ref{pre}). However, it remains open if general almost extremal NMDS codes support $t$-designs. It has been proved that the only extremal near-MDS codes with $q\ge 3$ and $k\ge 3$ are the $[12,6,6]$ extended Golay code and the codes derived from it \cite{Dodunekov1995} (see Section \ref{pre}). Moreover, $[2q+k,k,2q]$ extremal NMDS codes with $k> q$ and $[2q+k,k+1,2q-1]$ almost extremal NMDS codes with $k\ge q$ are all known \cite{Boer 1996}, as shown below.\par 
\begin{itemize}

	\item A $[2q+k,k,2q]$ extremal NMDS code over $\mathbb{F}_{q}$ with $k>q$ must be the dual of one of the codes in Table \ref{ex NMDS}.
	
	\begin{table}[h]
		\centering
		\setlength{\tabcolsep}{5mm}
		\caption{Dual codes of extremal NMDS codes with $k>q$}
		\label{ex NMDS}
		\begin{tabular}{|c|c|c|}
			\hline
			$q$ & Parameters & Descriptions \\ 
			\hline
			2 & [7,4,3] & Hamming code \\
			2 & [8,4,4] & extended Hamming code \\
			3 & [10,6,4] & punctured Golay code \\
			3 & [11,6,5] & Golay code \\
			3 & [12,6,6] & extended Golay code \\
			\hline
		\end{tabular}
	\end{table}
	
	\item A $[2q+k,k+1,2q-1]$ almost extremal NMDS code over $\mathbb{F}_{q}$ with $k\ge q$ must be the dual of one of the codes in Table \ref{al ex NMDS}.
	\begin{table}[h]
		\centering
		\setlength{\tabcolsep}{4mm}
		\caption{Dual codes of almost extremal NMDS codes with $k\ge q$}
		\label{al ex NMDS}
		\begin{tabular}{|c|c|c|}
			\hline
			$q$ & Parameters & Descriptions \\ 
			\hline
			2 & [6,3,3] & punctured Hamming code \\
			2 & [7,3,4] & Simplex code \\
			3 & [9,5,4] & shortened punctured Golay code \\
			3 & [10,5,5] & shortened Golay code \\
			3 & [11,5,6] & dual Golay code \\
			\hline
		\end{tabular}
	\end{table}
\end{itemize}
The existences and general constructions of $[2q+k,k,2q]$ extremal NMDS codes with $k\le q$ and $[2q+k,k+1,2q-1]$ almost extremal NMDS codes with $k< q$ are still open.
\subsection*{Motivation and objectives of this work}
To our best knowledge, the only known $q$-ary NMDS codes with length exceeding $q+1$ that simultaneously support $2$-designs are:
\begin{itemize}
	\item The extremal NMDS codes and their duals (listed in Table \ref{ex NMDS})
	\item The almost extremal NMDS codes and their duals (listed in Table \ref{al ex NMDS}) 
	\item Codes derived from them (see Section \ref{pre}, Corollary \ref{derived extremal})
\end{itemize}

Although these sporadic binary and ternary NMDS codes have been discovered, there remains a lack of general constructions for $q$-ary NMDS codes with length $n > q+1$ that can support $2$-designs. In this study, we develop the first generic construction of $q$-ary NMDS elliptic curve codes with length exceeding $q+1$ and simultaneously support $2$-designs. We use tools from both elliptic curves and subset sums in our construction. We first build a bridge between support designs of the minimum-weight codewords and designs from subset sums in the group of rational points. Next, we carefully choose some elliptic curves with special rational point groups and employ some known results of designs from subset sums in these finite abelian groups to derive codes supporting $2$-designs. In addition, we determine the weight distribution of the codes using tools from subset sums in finite abelian groups.
\subsection*{Organization of the Paper}
In Section \ref{pre} we first recall the famous Assmus-Mattson Theorem and some basic facts of NMDS codes. In Section \ref{elliptic curve codes} we present a generic construction of $q$-ary NMDS elliptic curve codes supporting $2$-designs whose lengths exceed $q+1$. In Section \ref{iv} we give the weight distribution of the constructed codes. Finally, Section \ref{Summary and Concluding Remarks} presents our summary and concluding remarks.

\section{PRELIMINARIES}\label{pre}
\subsection{The Assmus-Mattson Theorem and its generalized version}
The Assmus-Mattson Theorem is a useful tool to derive $t$-designs from linear codes, we begin by giving the general vision of the Assmus-Mattson Theorem.
\begin{theorem}[General version of Assmus-Mattson Theorem]\label{AM general}
	
	Let $\mathcal{C}$ be an $[n,k,d]$ linear code over $\mathbb{F}_{q}$. Let $t$  be an integer satisfying $1\leq t<\min\left\{d, d^{\perp}\right\}$. Assume that there are at most $d^{\perp}-t$ weights of $\mathcal{C}$ in $\{1,2,\ldots, n-t\}$.
	Then $\left(\mathcal{P}(\mathcal{C}),\mathcal{H}_{w}(\mathcal{C})\right)$
	and $\left(\mathcal{P}\left(\mathcal{C}^{\perp}\right),\mathcal{H}_{w}\left(\mathcal{C}^{\perp}\right)\right)$ are both $t$-designs for all $w\in\{0,1,\ldots, n\}$.
\end{theorem}
Note that Theorem \ref{AM general} applies to general $t$-designs, which may have repeated blocks or may be simple, as defined in Section \ref{introduction}. The following lemma gives a criterion for the block set $\mathcal{H}_{w}(\mathcal{C})$ to have no repeat blocks \cite{D2022}.

\begin{lemma}\label{AM lemma}
	Let $\mathcal{C}$ be an $[n,k,d]$ linear code over $\mathbb{F}_{q}$. Let $h$ be the largest integer satisfying $h\leq n$ and
	
	$$ h-\left\lfloor\frac{h+q-2}{q-1}\right\rfloor<d.$$
\end{lemma}
Then there are no repeated blocks in $\mathcal{H}_{w}(\mathcal{C})$ for any $d\leq w\leq h$.

Combining Theorem \ref{AM general} and Lemma \ref{AM lemma} together, we obtain the following well-known Assmus-Mattson Theorem \cite{Assmus 1969} for constructing simple $t$-designs from linear codes.

\begin{theorem}[Assmus-Mattson Theorem]\label{AM}
	
	Let $\mathcal{C}$ be an $[n,k,d]$ linear code over $\mathbb{F}_{q}$. Let $t$  be an integer satisfying $1\leq t<\min\left\{d, d^{\perp}\right\}$. Assume that there are at most $d^{\perp}-t$ weights of $\mathcal{C}$ in the range $\{1,2,\ldots, n-t\}$. Then $\left(\mathcal{P}(\mathcal{C}),\mathcal{H}_{w}(\mathcal{C})\right)$
	and $\left(\mathcal{P}\left(\mathcal{C}^{\perp}\right),\mathcal{H}_{w}\left(\mathcal{C}^{\perp}\right)\right)$ are both $t$-designs for all $w\in\{0,1,\ldots, n\}$. In particular,
	
	\begin{itemize}
		
		\item $\left(\mathcal{P}(\mathcal{C}),\mathcal{H}_{w}(\mathcal{C})\right)$ is a simple $t$-design provided that $A_{w}\neq 0$ and $d\leq w\leq h$, where $h$ is the largest integer satisfying $h\leq n$ and
		
		$$h-\left\lfloor\frac{h+q-2}{q-1}\right\rfloor<d.$$
		
		\item  $\left(\mathcal{P}\left(\mathcal{C}^{\perp}\right),\mathcal{H}_{w}\left(\mathcal{C}^{\perp}\right)\right)$ is a simple $t$-design provided that $A_{w}^{\perp}\neq 0$ and $d^{\perp}\leq w\leq h^{\perp}$, where
		$h^{\perp}$ is the largest integer satisfying $h^{\perp}\leq n$ and
		
		$$h^{\perp}-\left\lfloor\frac{h^{\perp}+q-2}{q-1}\right\rfloor<d^{\perp}.$$
	\end{itemize}
\end{theorem}
For a period of time, there were two major approaches of constructing $t$-designs, one was through the Assmus-Mattson Theorem \cite{D2022, D2018,DL2017,DT2020}, and another was employing linear codes with $t$-homogeneous or $t$-transitive automorphism groups \cite{D2022, Huffman 2003}. In order to provide a theoretical explanation to some $t$-designs from linear codes \cite{D2019 Vector, Tang Ding xiong 2019} that neither satisfy the conditions of the Assmus-Mattson theorem nor admit $t$-transitive or $t$-homogeneous automorphism groups, Tang, Ding and Xiong \cite{Tang Ding xiong 2020} developed a generalized version of the Assmus-Mattson theorem in 2020.

\begin{theorem}[Generalized Assmus-Mattson theorem]\label{AM generalized}
	
	Let $\mathcal{C}$ be an $[n,k,d]$ linear code over $\mathbb{F}_{q}$. Let $s$ and $t$ be two positive integers satisfy $t<\min\left\{d, d^{\perp}\right\}$. Let $S$ be a $s$-subset of $\{d,\ldots, n-t\}$. Assume that $(\mathcal{P}(C),\mathcal{H}_{\ell}(C))$ and $\left(\mathcal{P}\left(C^{\perp}\right),\mathcal{H}_{\ell^{\perp}}\left(C^{\perp}\right)\right)$ are $t$-designs for any $\ell\in\{d,\ldots, n-t\}\backslash S$ and $0\leq \ell^{\perp}\leq s+t-1$. Then $\left(\mathcal{P}(C),\mathcal{H}_{w}(C)\right)$ and $\left(\mathcal{P}\left(C^{\perp}\right),\mathcal{H}_{w}\left(C^{\perp}\right)\right)$ are $t$-designs for all $w\in\{0,1,\ldots, n\}$. In particular,
	
	\begin{itemize}
		\item $(\mathcal{P}(\mathcal{C}),\mathcal{H}_{w}(\mathcal{C}))$ is a simple $t$-design for any $w$ with $d\leq w\leq h$, where $h$ the largest integer satisfying $h\leq n$ and
		$$ h-\left\lfloor\frac{h+q-2}{q-1}\right\rfloor<d $$
		
		\item and $\left(\mathcal{P}\left(\mathcal{C}^{\perp}\right),\mathcal{H}_{w}\left(\mathcal{C}^{\perp}\right)\right)$ is a simple $t$-design for any $w$ with $d\leq w\leq h^{\perp}$, where $h^{\perp}$ is the largest integer satisfying $h^{\perp}\leq n$ and
		$$ h^{\perp}-\left\lfloor\frac{h^{\perp}+q-2}{q-1}\right\rfloor<d^{\perp}. $$
	\end{itemize}
\end{theorem}
Note that in Theorem \ref{AM generalized}, if the linear code $\mathcal{C}$ is NMDS, then we have the following corollary.
\begin{corollary}\label{coro AM near}
	Let $\mathcal{C}$ be a near MDS code with parameters $[n,k,n-k]$. If $\min\{k,n-k\}\ge 3$ and the minimum-weight codewords in $\mathcal{C}$ support a $2$-design, then $\left(\mathcal{P}(C),\mathcal{H}_{w}(C)\right)$ and $\left(\mathcal{P}\left(C^{\perp}\right),\mathcal{H}_{w}\left(C^{\perp}\right)\right)$ are $2$-designs for any $w$ with $2\leq w\leq n$.
\end{corollary}
\begin{IEEEproof}
	The parameters of $\mathcal{C}$ are $[n,k,n-k]$ and the parameters of the dual code $\mathcal{C}^\perp$ is $[n,n-k,k]$. Let $ S=\left \{n-k+1,\ldots,n-2 \right \}$, $s=k-2$. Let $t=2$, then $s+t-1=k-1\le k$, thus  $(\mathcal{P}(C),\mathcal{H}_{w}(C))$ and $\left(\mathcal{P}\left(C^{\perp}\right),\mathcal{H}_{w}\left(C^{\perp}\right)\right)$ are $2$-designs for all $w\in\{0,1,\ldots, n\}$.
\end{IEEEproof}
\subsection{NMDS codes and related $t$-designs}

An AMDS code has parameters of the form $[n,k,n-k]$. AMDS codes with dimensions 1, $n-2$, $n-1$, or $n$ are considered trivial. Since constructing trivial AMDS codes of arbitrary lengths is straightforward, we only focus on nontrivial AMDS codes. Unlike MDS codes, the dual of an AMDS code may not necessarily be AMDS. A code $\mathcal{C}$ is an NMDS code if both $\mathcal{C}$ and its dual $\mathcal{C}^\perp$ are AMDS. By definition, $\mathcal{C}$ is NMDS if and only if $\mathcal{C}^\perp$ is NMDS.

The following theorem outlines some fundamental characteristics of AMDS codes.

\begin{theorem}[\cite{Dodunekov1995}]
	For an AMDS code $\mathcal{C}$ with parameters $[n,k,n-k]$ over $\mathbb{F}_q$:
	\begin{itemize}
		\item When $k \geq 2$, the code length satisfies $n \leq k + 2q$.
		\item For $k \geq 2$ and $n-k > q$, the dimension is bounded by $k \leq 2q$.
		\item If $n-k > q$, the dual code $\mathcal{C}^\perp$ is also AMDS.
		\item For $k \geq 2$, the code is generated by codewords having weights $n-k$ and $n-k+1$.
		\item When $k \geq 2$ and $n-k > q$, the code is generated solely by its minimum weight codewords.
	\end{itemize}
\end{theorem}

The subsequent theorem provides a characterization of NMDS codes from their generator matrices.

\begin{theorem}[\cite{Dodunekov1995}]
	A code $\mathcal{C}$ with parameters $[n,k]$ over $\mathbb{F}_q$ is NMDS if and only if any generator matrix $G$ (and consequently every generator matrix of $\mathcal{C}$) satisfies:
	\begin{itemize}
		\item Any subset of $k-1$ columns in $G$ is linearly independent.
		\item There exists a set of $k$ columns that are linearly dependent.
		\item Every set of $k+1$ columns has rank $k$.
	\end{itemize}
\end{theorem}

The weight distributions of NMDS codes and their duals can be determined using the following formulas.

\begin{theorem}[\cite{Dodunekov1995}]\label{weight distribution NMDS}
	For an NMDS code $\mathcal{C}$ with parameters $[n,k,n-k]$, the weight
	distributions of $\mathcal{C}$ and $\mathcal{C}^\perp$ are given by
	\begin{align}
		A_{n-k+s}
		={}& \binom{n}{k-s}
		\sum_{j=0}^{s-1} (-1)^j \binom{n-k+s}{j} (q^{s-j}-1) \notag\\
		&\quad + (-1)^s \binom{k}{s} A_{n-k},
	\end{align}
	for $s \in \{1,2,\ldots,k\}$, and
	\begin{align}
		A_{k+s}^\perp
		={}& \binom{n}{k+s}
		\sum_{j=0}^{s-1} (-1)^j \binom{k+s}{j} (q^{s-j}-1) \notag\\
		&\quad + (-1)^s \binom{n-k}{s} A_k^\perp,
	\end{align}
	for $s \in \{1,2,\ldots,n-k\}$.
\end{theorem}
\begin{corollary}\label{min weight }
	For an NMDS code $\mathcal{C}$ with parameters $[n,k,n-k]$, $\mathcal{C}$ and $\mathcal{C}^\perp$ have the same number of minimum-weight codewords, i.e, $A_{n-k}=A_{k}^{\perp}$
\end{corollary}
Note that two NMDS codes with identical parameters $[n,k,n-k]$ over $\mathbb{F}_q$ may have different weight distributions, which represents a significant distinction between MDS and NMDS codes. The following theorem reveals a one-to-one correspondence between the minimum weight codewords in an NMDS code $\mathcal{C}$ and its dual code  $\mathcal{C}^\perp$.
\begin{theorem}[\cite{Faldum 1997}]\label{thm:near-mds-correspondence}
	For any near MDS code $\mathcal{C}$, there exists a natural correspondence between its minimum weight codewords and those of its dual $\mathcal{C}^\perp$, satisfying the following properties:
	\begin{itemize}
		\item For each minimum weight codeword $\mathbf{c} \in \mathcal{C}$, there exists a unique (up to a multiple),  minimum weight codeword $\mathbf{c}^\perp \in \mathcal{C}^\perp$ such that their supports are disjoint, i.e.,
		$$
		\mathrm{supp}(\mathbf{c}) \cap \mathrm{supp}(\mathbf{c}^\perp) = \emptyset.
		$$
		
		\item The number of minimum weight codewords in $\mathcal{C}$ equals that in $\mathcal{C}^\perp$.
	\end{itemize}
\end{theorem}

\begin{corollary}\label{NMDS complementary design}
	Let $\mathcal{C}$ be a near MDS code. If the minimum weight codewords of $\mathcal{C}$ form the blocks of a $t$-design, the minimum weight codewords of $\mathcal{C}^\perp$ form the complementary design. 
\end{corollary}
One of the most important problems with NMDS codes is the following: given $k$ and $q$, find the maximum value $m(q,k)$ for $n$, for which there exist an $q$-ary $[n,k,n-k]$ NMDS code.
\begin{proposition}\label{m(q,k) bound}
	$m(q,k)\le 2q+k$, and in case of equality $A_{n-k+1}=0$. 
\end{proposition}
NMDS codes meeting the above bound are called \emph{extremal}. For example, the ternary $[12,6,6]$ extended Golay code is an extremal NMDS code. The following Propositions shows that NMDS codes with special weight distributions may produce $t$-designs. In particular, extremal NMDS codes with $k\ge 3$ support $2$-designs. 
\begin{proposition}\label{extremal design}
	Let $\mathcal{C}$ be an $[n,k,n-k]$ NMDS code. If there is an integer $s$ such that $A_{n-k+s}=0$.  Then $\left(\mathcal{P}(\mathcal{C}),\mathcal{H}_{w}(\mathcal{C})\right)$
	and $\left(\mathcal{P}\left(\mathcal{C}^{\perp}\right),\mathcal{H}_{w}\left(\mathcal{C}^{\perp}\right)\right)$ are both $t$-designs for all $w\in\{0,1,\ldots, n\}$. In particular, any $[2q+k,k,2q]$ extremal NMDS code and its dual code support $(k-1)$-designs.
\end{proposition}
Proposition \ref{extremal design} can be obtained by Assmus-Mattson theorem \ref{AM} with $t=k-s$ and Proposition \ref{m(q,k) bound}. The following Proposition shows many NMDS codes can be derived from an NMDS code. 
\begin{proposition}\label{derived NMDS}
	The existence of a $q$-ary $[n,k,n-k]$ NMDS code implies
	\begin{itemize}
		\item The existence of a $q$-ary $[n-1,k-1, n-k ]$ NMDS code;
		\item The existence of a $q$-ary $[n-1,k,n-k-1]$ NMDS code. 
	\end{itemize}
\end{proposition}
The following corollary shows an extremal NMDS codes can derive extremal NMDS codes with lower dimensions.
\begin{corollary}\label{derived extremal}
	For any integer $t$ with $0\le t\le k$, we have $m(q,k)\le m(q,k-t)+t$. In particular, the existence of a $q$-ary extremal NMDS code of dimension $k$ implies the existence of $q$-ary extremal NMDS codes for any dimension $k'\le k$.
\end{corollary}

\subsection{Subset sums in finite abelian groups and derived $t$-designs}
We start by introducing the well-known \emph{subset-sum problem}, and then the $t$-designs from subset sums. Consider an finite abelian group $G$ written additively, and let $D \subseteq G$ be a finite subset with cardinality $n$. For any element $x \in G$, we define the counting function
$$
N(D,k,x) =\# \left\{T \subseteq D : \# T= k, \sum_{t \in T} t = x\right\},
$$
which enumerates all $k$-subsets of $D$ whose elements sum to $x$.

The \emph{subset-sum problem}, which asks whether $N(D,k,x) > 0$ for some $1 \leq k \leq n$, is an NP-complete problem with important applications in coding theory and cryptography \cite{P2023}. The problem's main difficulty stems from the arbitrary choice of subset $D$. However, significant progress has been made in the case $D = G$. In this paper we are only interested in the case $D=G$, and we list the following definitions and notations we will use.
\begin{itemize}
	\item $G^{*}:=G\setminus \left \{ 0 \right \} $.
	\item $\mathcal{B}_{k}^{x}:=\left\{T \subseteq G: \# T= k, \sum_{t \in T} t = x\right\},\quad b_{k}^{x}:=\#\mathcal{B}_{k}^{x}.$ \item $\mathcal{B}_{k}^{x,*}:=\left\{T \subseteq G^{*} : \# T= k, \sum_{t \in T} t = x\right\},\quad b_{k}^{x,*}:=\#\mathcal{B}_{k}^{x,*}.$
	\item $\exp(G)$ is the exponent of a group $G$, and for $x\in G$ define 
	$e(x):=\max\{d: d|\exp(G), x\in dG\}.$
	\item For any positive integer $d$, denote the $d$-torsion of $G$ by $G[d]:=\{g\in G:dg=0\}.$
	\item $\mu(\cdot )$ is the M\"{o}bius function. For a positive integer $n$ with prime factorization $n = \prod_{i=1}^{k} p_i^{e_i}$, the M\"obius function is defined as: $$
	\mu(n) = 
	\begin{cases}
		1 & \text{if } n = 1, \\
		(-1)^k & \text{if } e_i = 1 \text{ for all } i, \\
		0 & \text{if } e_i > 1 \text{ for any } i.
	\end{cases}
	$$
	\item For a prime $p$, and any integer $N$, we denote by $\nu_{p}(N)$ the $p$-adic valuation of $N$, defined as $$\nu_{p}(N)=\begin{cases}
		\max\{s\in\mathbb{N}^{+}:p^s\mid N\} & \text{ if } N\ne0, \\
		\infty  & \text{ if } N=0.
	\end{cases}$$
\end{itemize}For the additive group of a finite field $\mathbb{F}_{q}$ (an elementary abelian $p$-group), Li and Wan \cite{LW2008} obtained a formula for $b_{k}^{x}$. With a different character-theoretic approach, Kosters \cite{K2013} extended the formula of $b_{k}^{x}$ to arbitrary finite abelian groups, and also developed a similar formula for $b_{k}^{x,*}$. We present these formulas in the following theorem. 
\begin{theorem}\label{thm bb}
	Let $G$ be an abelian group of order $n$. For any $x\in G$ and
	$1\le k\le n$, we have
	\begin{align}
		b_k^x
		={}& \frac{1}{n}
		\sum_{s\mid \gcd(\exp(G),k)}
		(-1)^{k+k/s}
		\binom{n/s}{k/s} \notag\\
		& \sum_{d\mid \gcd(e(x),s)}
		\mu\!\left(\frac{s}{d}\right)\#G[d],
	\end{align}
	and for any $1\le k\le n-1$, we have
	\begin{align}
		b_k^{x,*}
		={}& \frac{1}{n}
		\sum_{s\mid \exp(G)}
		(-1)^{k+\lfloor k/s\rfloor}
		\binom{n/s-1}{\lfloor k/s\rfloor} \notag\\
		& \sum_{d\mid \gcd(e(x),s)}
		\mu\!\left(\frac{s}{d}\right)\#G[d].
	\end{align}
\end{theorem}
Moreover, it was found that in any elementary abelian $p$-group ($p$ odd), $b_{k}^{x}$ and $b_{k}^{x,*}$ are nonzero except some trivial cases, as shown in the following theorem \cite{K2013,FP2021Permutations }.
\begin{theorem}\label{thm b}
	Let $p$ be odd, and let $G=\mathbb{Z}_{p}^{m}=\underbrace{\mathbb{Z}_{p}\oplus\cdots\oplus\mathbb{Z}_{p}}_{\text{$m$ times}}$. For $1\leq k\leq\left|G\right|=p^m$ and $x\in G$, the following assertions are true:
	\begin{itemize}
		\item  For all $x$ in $G$, $b_{k}^{x}>0$, the only exception is the trivial case $k=|G|=p^{m}$ for $x\neq 0$.
		\item For all $x$ in $G$, $b_{k}^{x,*}>0$, the only exceptions are the trivial cases $k=1$ and $k=p^{m}-2$ for $x=0$, and $k=p^{m}-1$ for $x\neq 0$.
	\end{itemize}
\end{theorem}

The question of whether $(G, \mathcal{B}_{k}^{x})$ forms a $t$-design has strong connections to coding theory \cite{D1973}, particularly for Hamming codes when $G = \mathbb{Z}_{2}^{m}$ \cite{FP2021}. In \cite{FP2021} and \cite{P2023} the authors established necessary and sufficient conditions for $(\mathbb{Z}_{p}^{m}, \mathcal{B}_{k}^{x})$ to be $1$-designs and $2$-designs when $p$ is a prime. The following theorem \cite{P2023} presents necessary and sufficient conditions for $(\mathbb{Z}_{p}^{m}, \mathcal{B}_{k}^{x})$ to be $1$-designs and $2$-designs when $p$ is an odd prime.
\begin{theorem}\label{thm element p design}
	
	Let $G=\mathbb{Z}_{p}^{m}=\underbrace{\mathbb{Z}_{p}\oplus\cdots\oplus\mathbb{Z}_{p}}_{\text{$m$ times}}$, and let $x\in G$, $1\le k\le p^{m}-1$. Then
	$(G,\mathcal{B} _{k}^{x})$ is a $1$-design if and only if $p\mid k$. Furthermore, $(G,\mathcal{B} _{k}^{x} )$ is a $2$-design if and only if $p\mid k$ and $x=0$.
	
\end{theorem}
Very recently, the condition of $1$-design in Theorem \ref{thm element p design} has been extended to general finite abelian $p$-groups in \cite{Liu 2025}. Further, they conjectured that subset sums in finite non-elementary abelian $p$-groups do not hold $2$-designs.

\begin{theorem}\label{thm p}
	Let $p$ be an odd prime and let $G$ be a finite abelian $p$-group isomorphic to $ \mathbb{Z}_{p^{t_{1}}} \oplus \mathbb{Z}_{p^{t_{2}}} \oplus\cdots \oplus\mathbb{Z}_{p^{t_{m}}}$, where $1\le t_{1}\le \cdots\le t_{m}$. For any $1\le k<n$, and $x=(x_{1}, x_{2},\cdots, x_{m})\in G$, with $x_{i}\in \mathbb{Z}_{p^{t_{i}}}  $, $(G, \mathcal{B} _{k}^{x})$ is a $1$-design if and only if one of the following conditions holds:
	\begin{enumerate}
		\item[(i)] $p^{t_{m}}|k$, and $x\in G$ is an arbitrary element.
		\item[(ii)] There exist at least one $ i, 1\le i\le m$, such that $p\nmid x_{i}.$ 
		\item[(iii)] $p \mid x_{i}$ for all $i$ with $1\le i\le m$, and $ \max\left \{ \nu_{p}(k)-\nu_{p}^{i}(x_{i})\mid 1\le i\le m \right \} \ge1$, where $\nu_{p}^{i}$ is the $p$-adic valuation restricted to $\mathbb{Z}_{p^{t_{i}}}$ defined by
		$$\nu_{p}^{i}(x)=
		\begin{cases}
			\nu_{p}(x) & \text{ if } x\ne0\pmod {p^{t_{i}}},\\
			\infty  & \text{ if } x=0\pmod {p^{t_{i}}}.
		\end{cases}$$
	\end{enumerate}
\end{theorem}

\section{NMDS elliptic curve codes supporting $2$-designs}\label{elliptic curve codes}
In this section, we propose a generic construction of $q$-ary NMDS elliptic curve codes supporting $2$-designs, and show that a lot of NMDS codes from this construction have lengths exceeding $q+1$. In addition, these codes satisfy the Generalized Assmus-Mattson theorem but do not satisfy the Assmus-Mattson theorem. 

A natural way to construct NMDS codes of long length is through rational points on algebraic curves, these codes are called algebraic geometry (AG) codes. AG codes are natural generalizations of Reed-Solomon codes, and the elliptic curve codes are AG codes defined by algebraic curves of genus $g=1$. The following results shows that $q$-ary NMDS elliptic curve codes of long length (up to the number of rational points) do exist.
\begin{theorem}[\cite{Tsfasman1991}]
	Elliptic curve $[n, k, n-k]$ NMDS codes over $\mathbb{F}_{q}$, $q = p^t$, do exist for
	$$n \leq \begin{cases}
		q + \lceil 2\sqrt{q} \rceil & \text{if } p \text{ divides } \lceil 2\sqrt{q} \rceil \text{ and } t \text{ is odd,} \\
		q + \lceil 2\sqrt{q} \rceil + 1 & \text{otherwise,}
	\end{cases}$$
	and any $k \in \{2, 3, \ldots, n-2\}$.
\end{theorem}
We shall briefly cover some basic concepts of elliptic curves codes. An elliptic curve $E$ over a finite field $\mathbb{F}_q$ is defined by a nonsingular Weierstrass equation:
\begin{equation}\label{Weierstrass equation}
	E/\mathbb{F}_q: y^{2} + a_{1}xy + a_{3}y = x^{3} + a_{2}x^{2} + a_{4}x + b
\end{equation}
with coefficients $a_{i},b \in \mathbb{F}_q$. The canonical forms are classified by the characteristic $p$ and are simplified to the following forms:

\begin{enumerate}
	\item \textbf{$p \neq 2,3$}: $ y^{2} = x^{3} + a_{4}x + b $
	
	\item \textbf{$p = 2$}:
	\begin{itemize}
		\item Non-supersingular: 
		$ y^{2} + a_{1}xy = x^{3} + a_{2}x^{2} + b \quad (a_{1} \neq 0) $ 
		\item Supersingular:
		$ y^{2} + a_{3}y = x^{3} + a_{4}x + b \quad (a_{3} \neq 0) $ 
	\end{itemize}
	
	\item \textbf{$p = 3$}:
	$ y^{2} = x^{3} + a_{2}x^{2} + b $
	
\end{enumerate}

Fix an algebraic closure $\overline{\mathbb{F}}_{q} $ of $\mathbb{F}_{q}$, and \(\operatorname{Gal}(\overline{\mathbb{F}}_{q}/\mathbb{F}_{q})\) is the absolute Galois	group of $\mathbb{F}_{q}$. Let \(E(\mathbb{F}_{q})\) (resp. $E(\overline{\mathbb{F}}_{q})$) be the set of $\mathbb{F}_{q}$-rational points (resp. $\overline{\mathbb{F}}_{q}$-rational points) on \(E\), i.e., the set of solutions of E.q. (\ref{Weierstrass equation}) in $\mathbb{F}_{q}$ (resp. $\overline{\mathbb{F}}_{q}$) together with the infinity point $\infty$. Moreover, the set $E(\overline{\mathbb{F}}_{q})$ is naturally equipped with the action of \(\operatorname{Gal}(\overline{\mathbb{F}}_{q}/\mathbb{F}_{q})\). The set $E(\mathbb{F}_{q})$ forms a finite abelian group with the zero element $\infty$, and it has the following structure:
$$E(\mathbb{F}_q)\cong\mathbb{Z}_n\quad or\quad\mathbb{Z}_{n_1}\oplus\mathbb{Z}_{n_2},$$
for some integer $n$ or for some integers $n_{1},n_{2}$ with $n_{1}\mid n_{2}$. Further, The well-known \emph{Hasse-Weil bound} says that 
\begin{equation}\label{Hasse-Weil bound}
	|\#E(\mathbb{F}_{q}) - (q + 1)| \leq 2\sqrt{q}.
\end{equation}  \par 
The divisor group \(\operatorname{Div}(E)\) is a free abelian group generated by \(E(\overline{\mathbb{F}}_{q})\). Thus, a divisor \(D\) is a sum \(D = \sum n_{P}P\), where \(P\) runs through \(E(\overline{\mathbb{F}}_{q})\), and \(n_{P}\) is an integer of nonzero value for finitely many \(P\). The support of the divisor $D$, denoted by $\mathrm{supp}(D)$, is defined as the set of points with non-zero coefficients, i.e., $$\mathrm{supp}(D) = \{P \in E(\overline{\mathbb{F}}_{q}) \mid n_{P} \neq 0\}.$$ The action of \(\operatorname{Gal}(\overline{\mathbb{F}}_{q}/\mathbb{F}_{q})\) on \(E(\overline{\mathbb{F}}_{q})\) naturally extends to  \(\operatorname{Div}(E)\), and a divisor \(D = \sum n_{P}P\) is called an \(\mathbb{F}_{q}\)-rational divisor if it is invariant under \(\operatorname{Gal}(\overline{\mathbb{F}}_{q}/\mathbb{F}_{q})\). The degree of a divisor is defined to be a $\mathbb{Z}$-linear map $$\deg: \operatorname{Div}(E) \to  \mathbb{Z},$$ given by $\deg\left (  \sum n_{P}P\right ) =\sum n_{P}$. The kernel of the map is a subgroup of $\operatorname{Div}(E)$ and is denoted by $\operatorname{Div}^{0}(E)$. For more information about elliptic curves, the reader is referred to Silverman's book \cite{St2009}. Next, we will briefly introduce some basic concepts of function fields and elliptic curve codes, which are covered in the first two chapter of Stichtenoth's book \cite{S2009}.\par 
Denote by $\mathbb{F}_{q}(E):=\mathbb{F}_{q}(x,y)$ (resp. $\overline{\mathbb{F}}_{q}(E):=\overline{\mathbb{F}}_{q}(x,y)$) the function field generated by the elliptic curve $E$ over $\mathbb{F}_{q}$ (resp. $\overline{\mathbb{F}}_{q}$) , where $x$ and $y$ satisfy E.q. (\ref{Weierstrass equation}). For any $P\in E(\overline{\mathbb{F}}_{q}) $, denote its valuation by $$v_{P}:\overline{\mathbb{F}}_{q}(E)^{\times}\to \mathbb{Z}.$$ The \emph{principal divisor} of $f\in \mathbb{F}_{q}(E)^{\times}$ is defined by $\operatorname{div}(f)=\sum_{P} v_{P}P$, and the following set forms a subgroup of $\operatorname{Div}^{0}(E)$: $$\operatorname{Princ}(E)=\left \{  \operatorname{div}(f)=\sum_{P} v_{P}(f)P\mid f\in\overline{\mathbb{F}}_{q}(E)^{\times} \right \},$$ which is called the group of principal divisors. Two divisors $D_{1},D_{2}\in \operatorname{Div}(E)$ are referred to as \emph{equivalent} and denoted by $D_{1}\sim D_{2}$ if there exist a principal divisor $\operatorname{div}(f)\in \operatorname{Princ}(E)$, such that $D_{1}=D_{2}+\operatorname{div}(f)$. The factor group $$\operatorname{Cl}^{0}(E):= \operatorname{Div}^{0}(E)/ \operatorname{Princ}(E)$$ is called the \emph{divisor class group}. There is a well-known group isomorphism $$ \varphi:E(\overline{\mathbb{F}}_{q})\to \operatorname{Cl}^{0}(E),$$ given by $P\mapsto \left [ P-\infty\right ] $, where $\left [ P-\infty\right ]$ denotes the equivalence class of $P-\infty$. Then the group structure of $\operatorname{Cl}^{0}(E)$ is transferred to $E(\overline{\mathbb{F}}_{q})$, and we denote by $\oplus$ the plus operation in $E(\overline{\mathbb{F}}_{q})$. For any $\mathbb{F}_{q}$-rational divisor $D$,  $$\mathscr{L}(D)=\{f\in\mathbb{F}_{q}(E)^{\times} \mid \operatorname{div}(f)\ge -D \}\cup \{0\}$$ is the \emph{Riemann-Roch space} of $D$. $\mathscr{L}(D)$ is an $\mathbb{F}_{q}$-vector space, and by the Riemann-Roch theorem, if $\deg(D)>0$, then $\dim_{\mathbb{F}_{q}}(\mathscr{L}(D))=\deg (D) $. \par

\begin{definition}\label{def of ell code}
	Let the $\mathcal{P}=\left \{ P_{1},P_{2},\cdots, P_{n} \right \} $ be a subset of  \(E(\mathbb{F}_{q})\), and let $D$ be an \(\mathbb{F}_{q}\)-rational divisor on $E$ ($0<\deg(D)<n$) with $\mathrm{supp}(D)\cap\mathcal{P} =\emptyset$. The elliptic curve code $\mathcal{C}(E,\mathcal{P},D )$ is defined by the image of the evaluation map $\mathrm{ev}_{D}:\mathscr{L}(D)\to\mathbb{F}_{q}^{n} $, given by 
	$$\mathrm{ev}_{D}:f\mapsto  \left ( f(P_{1}), f(P_{2}),\cdots,f(P_{n}) \right ) \in\mathbb{F}_{q}^{n}. $$
\end{definition}
By the Riemann-Roch Theorem, the dimension $k=\dim_{\mathbb{F}_{q}}(\mathscr{L}(D))=\deg(D)$. The minimum distance of $C(E,\mathcal{P},D )$ is lower bounded by $n-\deg(D)$, for the number of zeros of any function in $\mathscr{L}(D)$ is upper bounded by $\deg(D)$, together with the Singleton bound, we have
\begin{equation}\label{bound}
	n-k\le d\le n-k+1.
\end{equation}

From E.q. (\ref{bound}), the minimum distance of the elliptic curve code $\mathcal{C}(E,\mathcal{P},D )$ is either $n-k+1$ or $n-k$. The former case corresponds to an MDS code, while the latter defines an NMDS code, since the dual of an elliptic curve code is also an elliptic curve code \cite{S2009}. Very recently, Han and Ren \cite{HR2024} proved that the maximal length of $q$-ary MDS elliptic curve codes is close to $(1/2+\epsilon)\#E(\mathbb{F}_{q})$, which confirmed a conjecture proposed by Li, Wan and Zhang in \cite{LWZ2015}. Denote by $\oplus$ the plus operator in the group $E(\overline{\mathbb{F}}_{q})$.

We quote a well-known lemma (see \cite{HR2024,St2009,S2009}), which provides a sufficient condition for an $\mathbb{F}_{q}$-rational divisor to be a principal divisor. 
\begin{lemma}\label{prin div}
	Let $E/\mathbb{F}_{q}$ be an elliptic curve with the point $\infty$ at infinity. If $G=\sum_{Q}n_{Q}Q$ is a $\mathbb{F}_{q}$-rational divisor with $\sum_{Q}n_{Q}=0$ and $\bigoplus_{Q} \left [  n_{Q}\right ] Q=\infty$, then $G$ is a principal divisor. In particular, there exist a function $f\in\mathbb{F}_{q}(E)^{\times}$ such that $\operatorname{div}(f)=G$.
\end{lemma}

From now on, for an $\mathbb{F}_{q}$-rational divisor $D=\sum n_{Q}Q$, define $$Q_{D}:=\bigoplus \left [  n_{Q}\right ] Q.$$ Denote the number of $k$-subsets of $\mathcal{P}$ that sum up to $Q_{D}$ by $$N(\mathcal{P},k,Q_{D}):=\#\left\{T\subseteq \mathcal{P}{:}\,\#T=k,\sum_{t\in T}t=Q_{D}\right\}.$$

In the following proposition we present a criterion for the code $\mathcal{C}(E,\mathcal{P},D )$ to be NMDS from subset sums in the group $E(\mathbb{F}_{q})$.
\begin{proposition}\label{elli MDS}
	Let $E/\mathbb{F}_{q}$ be an elliptic curve with group $E(\mathbb{F}_{q})$. Let $\mathcal{P}=\left \{ P_{1},P_{2},\cdots, P_{n} \right \} $ be a subset of \(E(\mathbb{F}_{q})\), and let $D=\sum n_{Q}Q$ be an \(\mathbb{F}_{q}\)-rational divisor, where $0<\deg (D)=k<n$, $\mathrm{supp}(D)\cap\mathcal{P} =\emptyset$. Then the $[n,k,d]$ elliptic curve code $\mathcal{C}(E,\mathcal{P},D )$  is MDS, i.e, $d=n-k+1$, if and only if $$N(\mathcal{P},k,Q_{D})=0;$$ $\mathcal{C}(E,\mathcal{P},D )$ is NMDS, i.e., $d=n-k$ if and only if $$N(\mathcal{P},k,Q_{D})> 0.$$
	
\end{proposition}
\begin{IEEEproof}
	The two statements are equivalent, so we only prove the latter one. If $N(\mathcal{P},k,Q_{D})> 0$, then there exists a set $\{P_{i_{1}},P_{i_{2}},\cdots, P_{i_{k}}\}\subseteq D$, with \begin{equation}\label{sum}
		P_{i_{1}}\oplus P_{i_{2}}\oplus\cdots\oplus P_{i_{k}}=Q_{D}.\end{equation}
	Let $G=\sum_{j=1}^{k}P_{ij}-\sum n_{Q}Q$, then $G$ is a $\mathbb{F}_{q}$-rational divisor satisfying the condition in Lemma \ref{prin div} by Eq. (\ref{sum}). By Lemma \ref{prin div}, that there is a function $f\in \mathbb{F}_{q}(E)^{\times}$, such that \begin{equation}\label{func}
		\mathrm{div}(f)=\sum_{j=1}^{k}P_{ij}-\sum n_{Q}Q
	\end{equation}
	The degree of any principal divisor is zero, so Eq. ($\ref{func}$) is equivalent to saying that there exists a function $f\in \mathbb{F}_{q}(E)^{\times}$, such that
	\begin{equation}
		\mathrm{div}(f)\ge \sum_{j=1}^{k}P_{ij}-\sum n_{Q}Q.
	\end{equation}
	That is, there exists a function $f\in \mathscr{L}(D)$ with exactly $k$ zeros at points $P_{i_{1}},\cdots, P_{i_{k}}$. This implies the existence of a codeword $c_{f}\in \mathcal{C}(E,\mathcal{P},D )$, where $c_{f}=ev_{D}(f)$, and $$\mathrm{supp}(c_{f})=\left \{  1,2,\cdots, n \right \} \setminus\left \{ i_{1},i_{2},\cdots, i_{k}  \right \},$$ and thus the minimum distance  $d=n-k$. Conversely, if $d=n-k$, then there exists a function $f\in \mathscr{L}(D)$ with exactly $k$ zeros at points $P_{i_{1}},\cdots, P_{i_{k}}$ in $\mathcal{P}$. Then $\mathrm{div}(f)\ge \sum_{j=1}^{k}P_{ij}-\sum n_{Q}Q$ and $P_{i_{1}}\oplus P_{i_{2}}\oplus\cdots\oplus P_{i_{k}}=Q_{D}$. Consequently, the elliptic code $\mathcal{C}(E,\mathcal{P},D )$ is NMDS if and only if $N(\mathcal{P},k,Q_{D})> 0$.
	
\end{IEEEproof}
Proposition \ref{elli MDS} connects NMDS property of elliptic curve code $\mathcal{C}(E,\mathcal{P},D )$ with subset sums in the group $E(\mathbb{F}_{q})$. Further, when the elliptic curve code $\mathcal{C}(E,\mathcal{P},D )$ is NMDS and $E(\mathbb{F}_{q})$ is a $p$-group, all $1$-designs and $2$-designs from the minimum-weight codewords in $\mathcal{C}(E,\mathcal{P},D )$ are given in the following theorem. 
\begin{theorem}\label{1,2}
	Let $E/\mathbb{F}_{q}$ be an elliptic curve, and $E(\mathbb{F}_{q})\cong\mathbb{Z}_{p^{e_{1}}}\oplus \mathbb{Z}_{p^{e_{2}}}$ ($1< e_{1}\le e_{2}$), where $p$ is an odd prime. Assume $E(\mathbb{F}_{q})=\{P_{1},\cdots,P_{n}\}$ with $n=p^{e_{1}+e_{2}}$, and define $\mathcal{P}=\{P_{1},\cdots,P_{n}\}$. Let $D=\sum n_{Q}Q$ be an \(\mathbb{F}_{q}\)-rational divisor, where $0<\deg (D)=k<n$, $\mathrm{supp}(D)\cap\mathcal{P} =\emptyset$. Assume $Q_{D}=\bigoplus \left [  n_{Q}\right ] Q$ corresponds to $x:=(x_{1},x_{2})\in \mathbb{Z}_{p^{e_{1}}}\oplus \mathbb{Z}_{p^{e_{2}}} $, where $0\le x_{i} \le p^{e_{i}}-1, i=1,2$. Then the elliptic curve code $\mathcal{C}(E,\mathcal{P},D )$ is NMDS, and
\begin{enumerate}[label=(\roman*)]
	\item If $e_1=e_2=1$, then the minimum-weight codewords in the
	$[p^2,k,p^2-k]$ code $\mathcal{C}(E,\mathcal{P},D)$ support a
	$1$-$(p^2,p^2-k,r)$ design if and only if $p\mid k$.
	
	\item If $e_2>1$, then the minimum-weight codewords in the
	$[p^{e_1+e_2},k,p^{e_1+e_2}-k]$ code $\mathcal{C}(E,\mathcal{P},D)$
	support a $1$-$(p^{e_1+e_2},p^{e_1+e_2}-k,r')$ design if and only if
	one of the following conditions holds:
	\begin{enumerate}[label=(\alph*)]
		\item $p^{e_2}\mid k$.
		\item $p\nmid x_1$ or $p\nmid x_2$.
		\item $p\mid x_i$ for $i=1,2$, and
		\[
		\max\bigl\{\nu_p(k)-\nu_p^i(x_i)\mid i=1,2\bigr\}\ge 1.
		\]
	\end{enumerate}
	
	\item If $e_1=e_2=1$, then the minimum-weight codewords in the
	$[p^2,k,p^2-k]$ code $\mathcal{C}(E,\mathcal{P},D)$ support a
	$2$-$(p^2,p^2-k,\lambda)$ design if and only if $p\mid k$ and $x=0$.
\end{enumerate}
\end{theorem}
\begin{IEEEproof}
	According to Proposition \ref{elli MDS}, the code $\mathcal{C}(E,\mathcal{P},D )$ is NMDS if and only if $N(E(\mathbb{F}_{q}), k, Q_{D})>0$. By Theorem \ref{thm b}, $b_{k}^{x}>0$ in the group $E(\mathbb{F}_{q})\cong\mathbb{Z}_{p^{e_{1}}}\oplus \mathbb{Z}_{p^{e_{2}}}$ ($p$ odd). Thus, the code $\mathcal{C}(E,\mathcal{P},D )$ is NMDS. Consider codewords of minimum weight $n-k$, and the incidence structure $\left ( \mathcal{P} (C) ), \mathcal{H}_{n-k}(C)\right )$, where $C:=\mathcal{C}(E,\mathcal{P},D )$. We claim that $\left ( \mathcal{P} (C) ), \mathcal{H}_{n-k}(C)\right )$ is a $t$-design if and only if $\left ( E(\mathbb{F}_{q}), \mathcal{B}_{k}^{x}\right )$ is a $t$-design. Indeed, consider the bijection $\Psi:\mathcal{B}_{n-k}^{-x}\to\mathcal{H}_{n-k}(C) $, given by 
	\[
	\begin{aligned}
		\Psi:\;&
		\bigl\{
		(P_{j_1},P_{j_2},\ldots,P_{j_{n-k}})
		\mid
		P_{j_1}\oplus P_{j_2}\oplus \cdots \oplus P_{j_{n-k}}
		=-Q_D
		\bigr\} \\
		&\longmapsto \{j_1,j_2,\ldots,j_{n-k}\}.
	\end{aligned}
	\]
	then the support $\{j_{1},j_{2},\cdots,j_{n-k}\}$ of a minimum-weight codeword $c\in \mathcal{C}(E,\mathcal{P},D ) $ corresponds to a unique block in $\mathcal{B}_{n-k}^{-x}$, with its complementary block in $\mathcal{B}_{k}^{x}$. Thus, $\left ( \mathcal{P} (C) ), \mathcal{H}_{n-k}(C)\right )$ is a $t$-design if and only if $\left ( E(\mathbb{F}_{q}), \mathcal{B}_{n-k}^{-x}\right )$ is a $t$-design,  equivalently, $\left ( E(\mathbb{F}_{q}), \mathcal{B}_{k}^{x}\right )$ is a $t$-design. By this correspondence, the support designs of $\mathcal{C}(E,\mathcal{P},D )$ are simply derived from $1$-designs and $2$-designs on $p$-groups given in Section \ref{pre}, Theorem \ref{thm element p design}.
\end{IEEEproof}

In the above theorem, we give support designs from the minimum-weight codewords in the elliptic curve code $\mathcal{C}(E,\mathcal{P},D )$ where $E(\mathbb{F}_{q})$ is a $p$-group. In the following lemma, we will show the existence of $\mathbb{F}_{q}$-rational divisor $D$, such that $Q_{D}=\infty$ ($x=0$), which matches the condition of $2$-design, given in the above Theorem \ref{1,2} (iii).
\begin{lemma}\label{divisor sum 0}
	Let $E/\mathbb{F}_{q}$ be an elliptic curve ($q\ge 7$), then there exists a point $Q=(x_{Q},y_{Q})\in E(\mathbb{F}_{q^{2}})\setminus E(\mathbb{F}_{q})$, such that $Q\oplus\phi(Q)=\infty$, where $ \phi $ is the Frobenius endomorphism: $(x,y)\mapsto (x^{q},y^{q})$. Consequently, $D=k\left(Q+\phi(Q)\right)$ is an $\mathbb{F}_{q}$-rational divisor with $Q_{D}=\infty$, for any positive integer $k$.
\end{lemma}
\begin{IEEEproof}
	The Weierstrass equation $E(x,y)$ of an elliptic curve $E$ is \begin{equation}
		y^{2}+a_{1}xy+a_{3}y=x^{3}+a_{2}x^{2}+a_{4}x+b.
	\end{equation}  First, we claim that there exists an $x_{Q}\in \mathbb{F}_{q} $ such that the quadratic equation $E(x_{Q},y)$ has no solution  $y$ in $\mathbb{F}_{q}$. Then we assume $y_{Q}\in \mathbb{F}_{q^{2}}$ in a solution of $E(x_{Q},y)$. In this way, we get a rational point $(x_{Q},y_{Q})\in  E(\mathbb{F}_{q^{2}})\setminus E(\mathbb{F}_{q}) $, with $x_{Q}\in \mathbb{F}_{q} $ and $y_{Q}\in \mathbb{F}_{q^{2}}\setminus \mathbb{F}_{q}$, which is the desired point $Q$ with $Q\oplus\phi(Q)=\infty$, shown as follows: apply the $q$-th Frobenius map to both sides of the equation $E(x_{Q},y_{Q})$:
	\begin{equation}\label{elli eqQ1}
		y_{Q}^{2}+a_{1}x_{Q}y_{Q}+a_{3}y_{Q}=x_{Q}^{3}+a_{2}x_{Q}^{2}+a_{4}x_{Q}+b.
	\end{equation} Then we obtain \begin{equation}\label{elli eqQ2}
		(y_{Q}^{q})^{2}+a_{1}x_{Q}y_{Q}^{q}+a_{3}y_{Q}^{q}=x_{Q}^{3}+a_{2}x_{Q}^{2}+a_{4}x_{Q}+b.
	\end{equation} 
	Combine Eqs.(\ref{elli eqQ1}) and (\ref{elli eqQ2}), we have
	\begin{equation}\label{elli eqQ3}
		y_{Q}^{2}+a_{1}x_{Q}y_{Q}+a_{3}y_{Q}=(y_{Q}^{q})^{2}+a_{1}x_{Q}y_{Q}^{q}+a_{3}y_{Q}^{q}.
	\end{equation}
	It follows that
	\begin{equation}\label{elli eqQ4}
		(y_{Q}^{q})^{2}-y_{Q}^{2}=a_{1}x_{Q}(y_{Q}-y_{Q}^{q})+a_{3}(y_{Q}-y_{Q}^{q}).
	\end{equation}
	Because $y_{Q}\in \mathbb{F}_{q^{2}}\setminus \mathbb{F}_{q}$, $y_{Q}^{q}\ne y_{Q}$, then 
	\begin{equation}\label{elli eqQ5}
		y_{Q}^{q}+y_{Q}=-a_{1}x_{Q}-a_{3}.
	\end{equation}
	Therefore, \begin{equation}\label{elli eqQ5}
		\begin{aligned}
			Q\oplus \phi(Q)
			&=(x_Q,y_Q)\oplus (x_Q^q,y_Q^q) \\
			&=(x_Q,y_Q)\oplus (x_Q,-y_Q-a_1x_Q-a_3)
			=\infty.
		\end{aligned}
	\end{equation}
	Now we complete the proof of the claim.
	\begin{case}\label{case char non 2}
		If $q$ is odd, then $a_{1}=a_{3}=0$, denote the right side of E.q.(\ref{elli eqQ1}) by $f(x)$. If there does not exist such a $x_{Q}$, that is, for any $x\in \mathbb{F}_{q}$, the equation $y^{2}=f(x)$ has two distinct solutions in $\mathbb{F}_{q}$. It follows that, when $q\ge 5$, $$\# E(\mathbb{F}_{q})=2q+1>q+1+2\sqrt{q},$$ then it contradicts the Hasse-Weil bound.
	\end{case}
	\begin{case}\label{case char=2}
		If $q$ is even, then either $a_{1}=1, a_{3}=0$ or $a_{1}=0, a_{3}\ne 0$. In the former case, $E(x,y)$ is 
		\begin{equation}\label{elli eqQ6}
			y^{2}+xy=x^{3}+a_{2}x^{2}+a_{4}x+b.
		\end{equation} Suppose to the contrary, then for any fixed $x\in \mathbb{F}_{q}^{*}$, the E.q.(\ref{elli eqQ6}) has two distinct solutions of $y$, while for $x=0$ it has a unique solution. It results that when $q\ge 7$
		$$\# E(\mathbb{F}_{q})=2(q-1)+1+1=2q>q+1+2\sqrt{q},$$
		contradicting the Hasse-Weil bound. The latter case leads to contradiction in a similar way, which is omitted for brevity.
	\end{case}
	
\end{IEEEproof}

In the following proposition, we give a sufficient condition for the structure of $E(\mathbb{F}_{q})$ to be $\mathbb{Z}_{p}\oplus\mathbb{Z}_{p}$. It follows from a well-known condition (see \cite{FRR2012}) for the group structure to be $\mathbb{Z}_{m}\oplus\mathbb{Z}_{n}$, by taking $m=n=p$.
\begin{proposition}\cite{FRR2012}\label{structure}
	Let $q$ be a prime power, $p$ is a prime, and let $t=p^{2}-q-1$. Then there is an elliptic curve $E/\mathbb{F}_{q}$ such that $E(\mathbb{F}_{q})\cong \mathbb{Z}_{p}\oplus\mathbb{Z}_{p} $ if $\gcd(t,q)=1$, $t^{2}\le 4q$, and $p\mid q-1$.
\end{proposition}

Combine Theorem \ref{1,2} (iii),  Lemma \ref{divisor sum 0} and Proposition \ref{structure} together, we obtain a generic construction of NMDS codes supporting $2$-designs. It turns out that these codes satisfy the Generalized Assmus-Mattson theorem but do not satisfy the Assmus-Mattson theorem.
\begin{theorem}\label{main construction}
	Let $q$ be a prime power ($q\ge 7$), and let $p$ be an odd prime, such that $\gcd(p^{2}-q-1,q)=1$, $-2\sqrt{q}\le p^{2}-q-1\le 2\sqrt{q}$, and $p\mid q-1$. For any integer $k$ satisfying $p \mid k $ and $ 0 < k < p^{2}/2 $, there exists an NMDS code with parameters $[p^{2},2k,p^{2}-2k]$, and its minimum-weight codewords support a $2$-$(p^{2},p^{2}-2k,\lambda)$ design, with $$\lambda=\frac{2\binom{p^{2}-2k}{2} \left [ \binom{p^{2}}{2k}+(p^{2}-1)\binom{p}{2k/p} \right ] }  {p^{4}(p^{2}-1)}, $$ and its dual code has parameters $[p^{2}, p^{2}-2k, 2k]$, whose minimum-weight codewords also support a $2$-$(p^{2},2k,\lambda^{\perp})$ design, where  $$\lambda^{\perp}=\frac{4k(2k-1)\binom{p^{2}-2k}{2} \left [ \binom{p^{2}}{2k}+(p^{2}-1)\binom{p}{2k/p} \right ] }  {p^{4}(p^{2}-1)(p^{2}-2k)(p^{2}-2k-1)}.$$ Furthermore, codewords of each fixed weight in both codes support a $2$-design. 
\end{theorem}
\begin{IEEEproof}
	Choose an elliptic curve $E/\mathbb{F}_{q}$ with $E(\mathbb{F}_{q})\cong \mathbb{Z}_{p}\oplus\mathbb{Z}_{p} $. Let $\mathcal{P}=E(\mathbb{F}_{q})=\{P_{1},\cdots,P_{p^{2}}\}$. For $k$ satisfying $p \mid k $ and $ 0 < k < p^{2}/2 $, choose an $\mathbb{F}_{q}$-rational divisor $D=k(Q+\phi(Q))$, with $Q=(x_{Q},y_{Q})\in E(\mathbb{F}_{q^{2}})\setminus E(\mathbb{F}_{q})$, such that $Q\oplus\phi(Q)=\infty$. Then we obtain an NMDS elliptic curve code with parameters $[p^{2}, 2k, p^{2}-2k]$. The existence of elliptic curve with such a rational point group is guaranteed by Proposition \ref{structure}, and the existence of such a $\mathbb{F}_{q}$-rational divisor $D$ is guaranteed by Lemma \ref{divisor sum 0}. Then the $2$-$(p^{2},p^{2}-2k,\lambda)$ design comes directly from Theorem \ref{1,2} (iii). According to Proposition \ref{elli MDS} and Theorem \ref{1,2}, the support of a minimum-weight codeword (up to a multiple) one-to-one corresponds to a block in $\mathcal{B}_{p^{2}-2k}^{0}$, thus the parameter 
	\begin{align*}
		\lambda=&\frac{b_{P^{2}-2k}^{0}\binom{p^{2}-2k}{2}}{\binom{p^{2}}{2}}\\=&\frac{b_{2k}^{0}\binom{p^{2}-2k}{2}}{\binom{p^{2}}{2}}\\=&\frac{2\binom{p^{2}-2k}{2} \left [ \binom{p^{2}}{2k}+(p^{2}-1)\binom{p}{2k/p} \right ] }  {p^{4}(p^{2}-1)}.
	\end{align*}
	The last equality comes from the formula in Theorem \ref{thm bb}. Denote $\mathcal{C}:= \mathcal{C}(E,\mathcal{P},D )$, observe that the code $(\mathcal{P}(\mathcal{C}),\mathcal{H}_{w}(\mathcal{C}))$ satisfies the Generalized Assmus-Mattson theorem (see Corollary \ref{coro AM near}) but does not satisfy the Assmus-Mattson theorem (see Theorem \ref{AM}). By Corollary \ref{coro AM near}, codewords of each fixed weight in $\mathcal{\mathcal{C}}(E,\mathcal{P},D )$ and $\mathcal{\mathcal{C}}(E,\mathcal{P},D )^{\perp}$ support a $2$-design. In particular, the minimum-weight codewords in the dual code $\mathcal{\mathcal{C}}(E,\mathcal{P},D )^{\perp}$ support the complementary design of the one supported by the minimum-weight codewords of $\mathcal{\mathcal{C}}(E,\mathcal{P},D )$ by Corollary \ref{NMDS complementary design}, thus \begin{align*}
		\lambda^{\perp}=&\frac{\lambda\binom{p^{2}-2}{p^{2}-2k}}{\binom{p^{2}-2}{p^{2}-2k-2}}\\=&\frac{4k(2k-1)\binom{p^{2}-2k}{2} \left [ \binom{p^{2}}{2k}+(p^{2}-1)\binom{p}{2k/p} \right ] }  {p^{4}(p^{2}-1)(p^{2}-2k)(p^{2}-2k-1)}.
	\end{align*}
\end{IEEEproof}
\begin{example}\label{example of elli 2-de}
	Let $q=7$ and $p=3$, and the elliptic curve $E/\mathbb{F}_{7}$: $y^{2}=x^{3}+2$ has the rational point group $E(\mathbb{F}_{7})\cong \mathbb{Z}_{3}\oplus\mathbb{Z}_{3} $. Let $\mathcal{P}=E(\mathbb{F}_{7})=\{P_{1},\cdots,P_{9}\}$. Let $\alpha$ be a root of $x^2+4$, then $\mathbb{F}_{49}=\mathbb{F}_{7}(\alpha)$. Choose $Q=(1,6\alpha)$ with $Q\oplus\phi(Q)=(1,6\alpha)\oplus(1,\alpha)=\infty$, and let $k=3$. Thus $D=3Q+3\phi(Q)$, with the Riemann-Roch space
\begin{align*}
	\mathscr{L}(D)
	=&\bigl\{f\in\mathbb{F}_{7}(E)^{\times}\mid \operatorname{div}(f)\ge -3(Q+\phi(Q))\bigr\}\cup\{0\} \\
	=&\operatorname{span}\Biggl\{
	1,\frac{1}{x-1},\frac{1}{(x-1)^2},\frac{1}{(x-1)^3}, \\
	&\qquad\qquad\frac{y}{(x-1)^2},\frac{y}{(x-1)^3}
	\Biggr\}.
\end{align*}
	The elliptic curve code $\mathcal{C}(E,\mathcal{P},D )$ with parameters $[9,6,3]$ is the image of the evaluation map $\mathrm{ev}_{D}:\mathscr{L}(D)\to\mathbb{F}_{7}^{9} $, given by 
	$$\mathrm{ev}_{D}:f\mapsto  \left ( f(P_{1}), f(P_{2}),\cdots,f(P_{9}) \right ) \in\mathbb{F}_{7}^{9}. $$
	Then a generator matrix $G$ of $\mathcal{C}(E,\mathcal{P},D )$ is
	$$G = \begin{bmatrix}
		1 & 1 & 1 & 1 & 1 & 1 & 1 & 1 & 1 \\
		0 & 6 & 6 & 4 & 4 & 2 & 2 & 3 & 3 \\
		0 & 3 & 4 & 2 & 5 & 4 & 3 & 2 & 5 \\
		0 & 1 & 1 & 2 & 2 & 4 & 4 & 2 & 2 \\
		0 & 4 & 3 & 1 & 6 & 1 & 6 & 6 & 1 \\
		0 & 6 & 6 & 1 & 1 & 1 & 1 & 6 & 6
	\end{bmatrix}_{\!6 \times 9}.$$ 
	According to our Magma computations, $G$ generates the NMDS code $\mathcal{C}(E,\mathcal{P},D )$ with parameters $[9,6,3]$ and weight enumerator  \[
	\begin{aligned}
		&1 + 72 z^3 + 324 z^4 + 3348 z^5 \\
		&\quad + 10656 z^6 + 30024 z^7 + 43794 z^8 + 29430 z^9.
	\end{aligned}
	\] Furthermore, the supports of minimum weight codewords in $\mathcal{C}(E,\mathcal{P},D )$ form a $2$-$(9,3,1)$ design, i.e., a Steiner triple system $S(2,3,9)$. The dual code of $\mathcal{C}(E,\mathcal{P},D )$ has parameters $[9,3,6]$ and weight enumerator $$1+72z^6+216z^8+54z^9.$$ Furthermore, the supports of minimum weight codewords in the dual code form a $2$-$(9,6,5)$ design.
\end{example}
In Table \ref{tab:NMDS codes}, we list more NMDS elliptic curve codes with concrete constructions generated by Theorem \ref{main construction}, where $\alpha$ denotes a root of the irreducible quadratic polynomial we choose in each case.
\begin{table}[htbp]
	
	\caption{NMDS elliptic curve codes supporting $2$-designs}
	\label{tab:NMDS codes}
	\centering
	\scriptsize
	\setlength{\tabcolsep}{2pt}
	\renewcommand{\arraystretch}{1.18}
	\resizebox{\columnwidth}{!}{%
		\begin{tabular}{c l l l l l}
			\toprule
			$\mathbb{F}_q$ & Elliptic curve & $E(\mathbb{F}_q)$ & Polynomial & Divisor $D$ & Code parameter \\
			\midrule
			$\mathbb{F}_7$ & $y^2=x^3+2$ & $\mathbb{Z}_3\oplus\mathbb{Z}_3$ & $x^2+4$ & $k[(1,6\alpha)+(1,\alpha)]$ & $[9,6,3]$ $(k=3)$ \\
			$\mathbb{F}_{13}$ & $y^2=x^3+3$ & $\mathbb{Z}_3\oplus\mathbb{Z}_3$ & $x^2+11$ & $k[(2,8\alpha)+(2,5\alpha)]$ & $[9,6,3]$ $(k=3)$ \\
			$\mathbb{F}_{31}$ & $y^2=x^3+11$ & $\mathbb{Z}_5\oplus\mathbb{Z}_5$ & $x^2+11$ & $k[(0,18\alpha)+(0,13\alpha)]$ & $[25,2k,25-2k]$ $(5\mid k)$ \\
			$\mathbb{F}_{43}$ & $y^2=x^3+3$ & $\mathbb{Z}_7\oplus\mathbb{Z}_7$ & $x^2+41$ & $k[(0,18\alpha)+(0,25\alpha)]$ & $[49,2k,49-2k]$ $(7\mid k)$ \\
			$\mathbb{F}_{157}$ & $y^2=x^3+15$ & $\mathbb{Z}_{13}\oplus\mathbb{Z}_{13}$ & $x^2+155$ & $k[(0,20\alpha)+(0,137\alpha)]$ & $[169,2k,169-2k]$ $(13\mid k)$ \\
			$\mathbb{F}_{307}$ & $y^2=x^3+14$ & $\mathbb{Z}_{17}\oplus\mathbb{Z}_{17}$ & $x^2+305$ & $k[(0,43\alpha)+(0,264\alpha)]$ & $[289,2k,289-2k]$ $(17\mid k)$ \\
			$\mathbb{F}_{3541}$ & $y^2=x^3+127$ & $\mathbb{Z}_{59}\oplus\mathbb{Z}_{59}$ & $x^2+3539$ & $k[(0,2484\alpha)+(0,1057\alpha)]$ & $[3481,2k,3481-2k]$ $(59\mid k)$ \\
			$\mathbb{F}_{4423}$ & $y^2=x^3+2811$ & $\mathbb{Z}_{67}\oplus\mathbb{Z}_{67}$ & $x^2+4420$ & $k[(0,3209\alpha)+(0,1214\alpha)]$ & $[4489,2k,4489-2k]$ $(67\mid k)$ \\
			$\mathbb{F}_{5113}$ & $y^2=x^3+381$ & $\mathbb{Z}_{71}\oplus\mathbb{Z}_{71}$ & $x^2+5108$ & $k[(0,2490\alpha)+(0,2623\alpha)]$ & $[5041,2k,5041-2k]$ $(71\mid k)$ \\
			\bottomrule
		\end{tabular}%
	}
	\end{table}
\begin{corollary}\label{q+1}
	Let $q$ be a prime power ($q\ge 7$), and let $p$ be an odd prime, such that $\gcd(p^{2}-q-1,q)=1$, $1\le p^{2}-q-1\le 2\sqrt{q}$, and $p\mid q-1$. For any integer $k$ satisfying $p \mid k $ and $ 0 < k < p^{2}/2 $, there exists $q$-ary NMDS codes supporting $2$-designs whose lengths exceed $q+1$.
\end{corollary}
\begin{IEEEproof}
	The corollary can be directly obtained from Theorem \ref{main construction}. 
\end{IEEEproof}
There are a lot of parameters $(q,p,t)$ satisfying the condition in Corollary \ref{q+1} and $t>0$, thus providing  many NMDS codes with length exceeding $q+1$ supporting $2$-designs. In the range $p<2000$, we find 38 such parameters $(q,p,t)$ by computer search, we list these parameters as well as their corresponding code parameters in Table \ref{qpt}. Further, we conjecture that there are infinitely many such parameters $(q,p,t)$, thus providing infinitely many such NMDS codes.
\begin{table}[htbp]
	\tiny
	\centering
	\setlength{\tabcolsep}{2.0mm} 
	\renewcommand{\arraystretch}{1.2} 
	\caption{$(q, p, t)$ satisfying corollary \ref{q+1} with corresponding NMDS code parameters}
	\label{qpt}
	\begin{tabular}{cccc}
		\toprule
		$q$ (prime powers) & $p$ (odd primes) & $t = p^2-q-1$ & Code Parameters ($p \mid k , 0 < k < p^{2}/2 $) \\
		\midrule
		7 & 3 & 1 & $[9,2k,9-2k]$ \\
		43 & 7 & 5 & $[49,2k,49-2k]$ \\
		157 & 13 & 11 & $[169,2k,169-2k]$ \\
		343 & 19 & 17 & $[361,2k,361-2k]$ \\
		4423 & 67 & 65 & $[4489,2k,4489-2k]$ \\
		6163 & 79 & 77 & $[6241,2k,6241-2k]$ \\
		19183 & 139 & 137 & $[19321,2k,19321-2k]$ \\
		22651 & 151 & 149 & $[22801,2k,22801-2k]$ \\
		26407 & 163 & 161 & $[26569,2k,26569-2k]$ \\
		37057 & 193 & 191 & $[37249,2k,37249-2k]$ \\
		113233 & 337 & 335 & $[113569,2k,113569-2k]$ \\
		121453 & 349 & 347 & $[121801,2k,121801-2k]$ \\
		143263 & 379 & 377 & $[143641,2k,143641-2k]$ \\
		208393 & 457 & 455 & $[208849,2k,208849-2k]$ \\
		292141 & 541 & 539 & $[292681,2k,292681-2k]$ \\
		375157 & 613 & 611 & $[375769,2k,375769-2k]$ \\
		412807 & 643 & 641 & $[413449,2k,413449-2k]$ \\
		527803 & 727 & 725 & $[528529,2k,528529-2k]$ \\
		590593 & 769 & 767 & $[591361,2k,591361-2k]$ \\
		843643 & 919 & 917 & $[844561,2k,844561-2k]$ \\
		981091 & 991 & 989 & $[982081,2k,982081-2k]$ \\
		1041421 & 1021 & 1019 & $[1042441,2k,1042441-2k]$ \\
		1193557 & 1093 & 1091 & $[1194649,2k,1194649-2k]$ \\
		1246573 & 1117 & 1115 & $[1247689,2k,1247689-2k]$ \\
		1441201 & 1201 & 1199 & $[1442401,2k,1442401-2k]$ \\
		1514131 & 1231 & 1229 & $[1515361,2k,1515361-2k]$ \\
		1905781 & 1381 & 1379 & $[1907161,2k,1907161-2k]$ \\
		2023507 & 1423 & 1421 & $[2024929,2k,2024929-2k]$ \\
		2397853 & 1549 & 1547 & $[2399401,2k,2399401-2k]$ \\
		2453923 & 1567 & 1565 & $[2455489,2k,2455489-2k]$ \\
		2548813 & 1597 & 1595 & $[2550409,2k,2550409-2k]$ \\
		2626021 & 1621 & 1619 & $[2627641,2k,2627641-2k]$ \\
		2864557 & 1693 & 1691 & $[2866249,2k,2866249-2k]$ \\
		3050263 & 1747 & 1745 & $[3052009,2k,3052009-2k]$ \\
		3198733 & 1789 & 1787 & $[3200521,2k,3200521-2k]$ \\
		3241801 & 1801 & 1799 & $[3243601,2k,3243601-2k]$ \\
		3734557 & 1933 & 1931 & $[3736489,2k,3736489-2k]$ \\
		3946183 & 1987 & 1985 & $[3948169,2k,3948169-2k]$ \\
		\bottomrule
	\end{tabular}
\end{table}

\section {The weight distribution of the code}\label{iv}
In this section we determine the weight distribution of the constructed code. Notice that in Example \ref{example of elli 2-de}, in the elliptic curve code $\mathcal{C}(E,\mathcal{P},D )$ we have $A_{w}>0$ for all $d\le w\le n$. The following proposition shows it actually holds for any code from the construction in Theorem \ref{main construction}.
\begin{proposition}
	For any NMDS elliptic curve code $\mathcal{C}(E,\mathcal{P},D )$ with parameters $[p^{2},2k,p^{2}-2k]$ in Theorem \ref{main construction}, we have $A_{w}>0$ for all $p^{2}-2k\le w\le p^{2}$.
\end{proposition}
\begin{IEEEproof}
	We only need to prove that for any $w$ with $p^{2}-2k\le w\le p^{2}$, there exist a function $f\in \mathscr{L}(D)$ with exactly $p^{2}-w$ zeros, then its corresponding codeword has weight $w$. If $w$ is odd, then $p^{2}-w$ is even. Then there exist a subset $\left \{ P_{i_{1}},\cdots, P_{i_{p^{2}-w}}\right \} \subseteq  \mathcal{P}=E(\mathbb{F}_{q})$, such that $$P_{i_{1}}\oplus\cdots\oplus P_{i_{p^{2}-w}}=\left[\frac{p^{2}-w}{2}\right]\left(Q\oplus \phi(Q)\right)=\infty.$$ This is due to $b_{p^{2}-w}^{0}>0$ in the $E(\mathbb{F}_{q})\cong \mathbb{Z}_{p}\oplus\mathbb{Z}_{p} $ (see Theorem \ref{thm b}). By Lemma \ref{prin div}, there exists a function $f\in \mathscr{L}(D)$ with zeros at $\left \{ P_{i_{1}},\cdots, P_{i_{p^{2}-w}}\right \}$. If $w$ is even, then $p^{2}-w$ is odd, let $p^{2}-w=2m+1$ ($m\ge 0$). According to Theorem \ref{thm b}, $b_{2m}^{0,*}>0$ in $E(\mathbb{F}_{q})\cong \mathbb{Z}_{p}\oplus\mathbb{Z}_{p} $, then there exist a set $\left \{ P_{j_{1}},\cdots, P_{j_{2m}}\right \}\subseteq E(\mathbb{F}_{q})\setminus \infty $, such that $$P_{j_{1}}\oplus\cdots\oplus  P_{j_{2m}}\oplus\left[2\right]\infty=\left[(m+1)\right]\left(Q\oplus \phi(Q)\right)=\infty,$$ then there exists a function $f\in \mathscr{L}(D)$ with zeros at $\left \{ P_{j_{1}},\cdots, P_{j_{2m}}, \infty\right \}$.
	
\end{IEEEproof}
The following theorem gives the weight distribution of the NMDS elliptic curve code $\mathcal{C}(E,\mathcal{P},D )$ and its dual code $\mathcal{C}(E,\mathcal{P},D )^{\perp}$ in Theorem \ref{main construction}. 
\begin{theorem}\label{min- -weight of ellipt code}
	For an NMDS elliptic curve code $\mathcal{C}(E,\mathcal{P},D)$ with
	parameters $[p^2,2k,p^2-2k]$ in Theorem~\ref{main construction}, we have
	\begin{equation}
		\label{eq:min-weight-A}
		A_{p^2-2k}=A_{2k}^{\perp}
		=\frac{(q-1)\left[\binom{p^2}{2k}+(p^2-1)\binom{p}{2k/p}\right]}{p^2},
	\end{equation}
	and
	\begin{align}
		A_{p^2-2k+s}
		={}& \binom{p^2}{2k-s}
		\sum_{j=0}^{s-1}(-1)^j
		\binom{p^2-2k+s}{j}(q^{s-j}-1)\notag\\
		&\quad +(-1)^s\binom{2k}{s}A_{p^2-2k},
	\end{align}
	for $s\in\{1,2,\ldots,2k\}$, and
	\begin{align}
		A_{2k+s}^{\perp}
		={}& \binom{p^2}{2k+s}
		\sum_{j=0}^{s-1}(-1)^j
		\binom{2k+s}{j}(q^{s-j}-1)\notag\\
		&\quad +(-1)^s\binom{p^2-2k}{s}A_{2k}^{\perp},
	\end{align}
	for $s\in\{1,2,\ldots,p^2-2k\}$.
\end{theorem}
\begin{IEEEproof}
	By Theorem \ref{1,2}, each block in $\mathcal{B}_{P^{2}-2k}^{0}$ one-to-one corresponds to a minimum-weight codewords (up to a multiple). By Corollary \ref{min weight } $A_{p^{2}-2k}=A_{2k}^{\perp}=(q-1)b_{p^{2}-2k}^{0}$, then the weight distribution is obtained by the formulas in Theorem \ref{weight distribution NMDS} and Theorem \ref{thm bb}.
\end{IEEEproof}

\section{ Summary and Concluding Remarks}\label{Summary and Concluding Remarks} 

This paper presents the first \emph{generic construction} of $q$-ary near-MDS (NMDS) codes supporting $2$-designs with lengths exceeding $q + 1$. Leveraging the connection between minimum-weight codewords in elliptic curve codes and subset sums in the rational point groups of elliptic curves, we constructed $2$-designs from those supported by subset sums in finite abelian groups. Notably, the resulting $2$-designs include new families of combinatorial designs that fall outside the reach of the classical Assmus–Mattson theorem. Our construction significantly expands the known landscape of $q$-ary NMDS codes supporting $2$-designs, particularly in the regime of lengths beyond the classical $q + 1$ barrier.

Since algebraic-geometry-based methods can construct $q$-ary NMDS codes with lengths up to approximately $q + 2\sqrt{q} + 1$, it is natural to ask whether more $t$-designs—possibly with varying parameters—can be obtained from such long NMDS codes. In particular, employing elliptic curves with different types of rational point groups may yield richer design structures. Exploring these directions remains an interesting avenue for future research.

\section*{Acknowledgments}
The authors are very grateful to the Associate Editor, Prof. Simeon Ball, and the anonymous reviewers for their valuable comments and suggestions that improved the presentation and quality of this article.

\ifCLASSOPTIONcaptionsoff
  \newpage
\fi



%

%

\begin{IEEEbiographynophoto}{Hengfeng Liu}
received the B.S. degree from the School of Mathematics, Southwest Jiaotong University, Chengdu, China, in 2024. He is currently pursuing a Ph.D. degree in Mathematics at Southwest Jiaotong University. His research interests include coding theory, cryptography, and algebraic combinatorics.
\end{IEEEbiographynophoto}


\begin{IEEEbiographynophoto}{Chunming Tang} (Member, IEEE) received the B.S. degree from Sichuan Normal University, Sichuan, in 2004, and the M.S. and Ph.D. degrees from Peking University, Beijing, China, in 2012. From 2017 to 2018, he was a Post-Doctoral Member with the Department of Mathematics, Universities of Paris VIII. From 2018 to 2020, he was a Post-Doctoral Member with the Department of Computer Science and Engineering, The Hong Kong University of Science and Technology. He was a Professor with the School of Mathematics and Information, China West Normal University, Nanchong, Sichuan. He is currently a Professor with the School of Information Science and Technology, Southwest Jiaotong University, Chengdu, China. His research interests include cryptography, coding theory, and information security.
\end{IEEEbiographynophoto}

\begin{IEEEbiographynophoto}{Zhengchun Zhou} (Senior Member, IEEE) received the B.S. and M.S. degrees in mathematics and the Ph.D. degree in information security from Southwest Jiaotong University, Chengdu, China, in 2001, 2004, and 2010, respectively. From 2012 to 2013, he was a Post-Doctoral Researcher with the Department of Computer Science and Engineering, The Hong Kong University of Science and Technology. From 2013 to 2014, he was a Research Associate with the Department of Computer Science and Engineering, The Hong Kong University of Science and Technology. He is currently a Professor with the School of Information Science and Technology (and also a Professor with the School of Mathematics), Southwest Jiaotong University. His research interests include coding theory, cryptography, and intelligent information processing. He was a recipient of the National Excellent Doctoral Dissertation Award in 2013 (China). He is an Associate Editor of several journals, including IEEE TRANSACTIONS ON INFORMATION THEORY, IEEE TRANSACTIONS ON COMMUNICATIONS, IEEE TRANSACTIONS ON COGNITIVE COMMUNICATIONS AND NETWORKING, Cryptography and Communications, and Advances in Mathematics of Communications.
	
\end{IEEEbiographynophoto}

\begin{IEEEbiographynophoto}{Dongchun Han} received the B.S. and M.S. degrees from Sichuan University, Chengdu, China, in 2009 and 2012, respectively, and the Ph.D. degree from Nankai University, Tianjin, China, in 2015, both in mathematics. He is currently a  Professor with the School of Mathematics, Southwest Jiaotong University. His research interests include combinatorics, number theory, and coding theory.
\end{IEEEbiographynophoto}

\begin{IEEEbiographynophoto}{Hao Chen} received the Ph.D. degree in mathematics from the Institute of Mathematics, Fudan University, in 1991. He is currently a Professor with the College of Information Science and Technology, Jinan University, China. His research interests include coding theory and cryptography, quantum information and computation, lattices, and algebraic geometry.
	
\end{IEEEbiographynophoto}




\end{document}